\documentclass[12pt,a4paper,twoside]{article}
\usepackage{fancyhdr}

\usepackage{latexsym}
\usepackage{amsmath}
\usepackage{verbatim}
\usepackage{amsthm}
\usepackage{amssymb}
\def\R{{\ifmmode{\rm I}\mkern-4mu{\rm R}
\else\leavevmode\hbox{j}\kern-.17em\hbox{R}\fi}}
\def\N{{\ifmmode{\rm I}\mkern-3.5mu{\rm N}
\else\leavevmode\hbox{j}\kern-.16em \hbox{N}\fi} }
\def\C{\ifmmode{{\rm C}\mkern-15mu{\phantom{\rm t}\vrule}}\mkern10mu
\else\leavevmode\hbox{C}\kern-.5em\hbox{j}\kern.3em\fi}

\def\Q{{\ifmmode{\rm I}\mkern-7.5mu{\rm Q}
\else\leavevmode\hbox{j}\kern-.17em\hbox{Q}\fi}}

\renewcommand{\d}{{\rm d}}
\renewcommand{\phi}{\varphi}

\renewcommand{\epsilon}{\varepsilon}

\renewcommand{\P}{\mathbb{P}}
\newcommand{\E}{\mathbb{E}}
\newcounter{numerator}%

\newtheorem{theorem}{Theorem}[section]

\newtheorem{proposition}[theorem]{Proposition}
\newtheorem{lemma}[theorem]{Lemma}

\begin{document}

\title{Random walks conditioned to stay in Weyl chambers\\ of type C and D}

\author{Wolfgang K\"onig and Patrick Schmid}

\date{November 3, 2009}

\maketitle

\abstract{We construct the conditional versions of a multidimensional random
walk given that it does not leave the Weyl chambers of type C and of type D,
respectively, in terms of a Doob $h$-transform. Furthermore, we prove functional
limit theorems for the rescaled random walks. This is an extension of recent
work by Eichelsbacher and K\"onig who studied the analogous conditioning for the
Weyl chamber of type A. Our proof follows recent work by Denisov and Wachtel who
used martingale properties and a strong approximation of random walks by
Brownian motion. Therefore, we are able to keep minimal moment assumptions.
Finally, we present an alternate function that is amenable to an $h$-transform
in the Weyl chamber of type C.}

\section{Introduction}\label{sec-Intro}
In his classical work \cite{Dys} Dyson established a connection between
dynamical versions of random matrices and non-colliding random particle systems.
Indeed, the eigenvalue process of a $k\times k$ Hermitian Brownian motion has
the same distribution as the evolution of $k$ independent standard Brownian
motions conditioned never to collide (which means that they are in the same
order at all times). This process can also be characterised by saying that a
$k$-dimensional Brownian motion is conditioned on never leaving the Weyl chamber
of type A, $W^{\rm A} = \{x = (x_1, \ldots, x_k) \in \mathbb{R}^k \colon x_1 <
\ldots < x_k \}$. 

This conditional process, called {\it Dyson's Brownian motion}, attracted the
interest of various researchers. Several discrete versions were considered.
Recently, Eichelsbacher and K\"onig \cite{EiKoe} constructed, in great
generality, the analogous random walk version, i.e., the conditional version of
a  random walk on $\mathbb{R}^k$ given that it never leaves $W^{\rm A}$. This
result and its proof were recently improved by Denisov and Wachtel
\cite{DenWac}. It is the aim of this paper to extend their analysis to the two
cases of the Weyl chambers of type C and D, see Section~\ref{sec-CandD}.

Let us first describe the random walk version for the type-A chamber. To fix
notation, let $S(n)= (S_1(n), \ldots, S_k(n))$ denote the position of a random
walk in $\mathbb{R}^k$ started at $x\in\mathbb{R}^k$ after $n$ steps with
components $S_j(n) = x_j + \xi_j^{(1)} + \ldots + \xi_j^{(n)}, 1\leq j \leq k$,
where $\{\xi_j^{(i)}\colon 1\leq j \leq k, i\in\N \}$ is a family of independent
identically distributed random variables. In particular, $S(0) = x$. We write
$\P_x$ and $\E_x$ for the corresponding probability measure and expectation.

Actually one can understand conditioning to never leave $W^{\rm A}$ in two ways.
If $\tau_x^{\rm A} = \inf\{n\in\N_0\colon S(n) \notin W^A \}$ denotes the exit
time from $W^{\rm A}$, then on the one hand one can mean the conditional
distribution of the path given the event $\{\tau_x^{\rm A} > m\}$ asymptotically
as $m$ grows to infinity, that is, 
$$
\widehat{\mathbb{P}}_x(S(n) \in \d y) = \lim_{m\rightarrow\infty}
\mathbb{P}_x(S(n)\in \d y\mid \tau_x^{\rm A} > m), \qquad x,y\in W^{\rm A}.
$$
On the other hand, one can make a change of measure by Doob's $h$-transform
\cite{Do}. Necessary for this procedure is to find a function $h$ which is
strictly positive on $W^{\rm A}$ and regular for the restriction of the
transition kernel to $W^{\rm A}$, i. e., 
$$
\mathbb{E}_x[h(S(1)); \tau_x^{\rm A} > 1] = h(x), \qquad x\in W^{\rm A}.
$$ 
Then a new probability transition function on $W^{\rm A}$ is defined by
$$
\widehat{\mathbb{P}}_x^{(h)}(S(n) \in \d y) = \mathbb{P}_x(S(n) \in \d y;
\tau_x^{\rm A} > n)\frac{h(y)}{h(x)}, \qquad x,y\in W^{\rm A}.
$$
The corresponding Markov chain is called the $h$-transform on $W^{\rm A}$. A
priori there may be more than one function $h$ amenable to this procedure.
However, if a positive regular function $h$ governs the upper tails of
$\tau_x^{\rm A}$, i.e., $\mathbb{P}(\tau_x^{\rm A} > n) \sim C_1 h(x) n^{-c_2}$
as $n\to\infty$ for some $C_1, c_2 >0$ for any $x\in W^{\rm A}$, 
then the two above constructions lead to the same process. Indeed, by the Markov
property one obtains in the limit $m\to\infty$
$$
\begin{aligned}
\mathbb{P}_x(S(n) \in \d y \mid \tau_x^{\rm A} >m) & =  \mathbb{P}_x(S(n)\in \d
y; \tau_x^{\rm A} > n)\frac{\mathbb{P}(\tau_y^{\rm A} > m -
n)}{\mathbb{P}(\tau_x^{\rm A} > m)} \\
 & \rightarrow \mathbb{P}_x(S(n)\in \d y; \tau_x^{\rm A} > n)\frac{h(y)}{h(x)}.
\end{aligned}
$$

Eichelsbacher and K\"onig succeeded in finding a positive regular function
$V^{\rm A}$ which yields this coincidence:
\begin{equation}\label{VAdef}
V^{\rm A}(x) = h^{\rm A}(x) - \mathbb{E}_x[h^{\rm A}(S(\tau_x^{\rm A}))], \qquad
x\in W^{\rm A},
\end{equation}
with $h^{\rm A}$ the {\it Vandermonde determinant}
$$
h^{\rm A}(x) = \prod_{1\leq i < j \leq k}(x_j - x_i) =
\det\left((x_i^{j-1})_{i,j\in \{1,\dots, k\}}\right).
$$
It should be noted that $h^{\rm A}$ is, up to a multiplicative positive
constant, the unique harmonic function that is positive on the interior of
$W^{\rm A}$ and vanishes on the boundary. In potential theoretic terms, this is
expressed by saying that $h^{\rm A}$ is the {\it r\'{e}duite} of $W^{\rm A}$. 
Furthermore, Eichelsbacher and K\"onig showed that the rescaled random walk
weakly converges to Dyson's Brownian motion. 

It is not easy to see that $V^{\rm A}$ is strictly positive on $W^{\rm A}$, and
it is surprisingly difficult to prove that $V^{\rm A}$ is well-defined, i.e.,
that $h^{\rm A}(S(\tau_x^{\rm A}))$ is integrable. The approach in \cite{EiKoe}
is based on the discrete analogue of the Karlin-McGregor formula \cite{KarMc} 
for random walks and an application of a local central limit theorem. By
repeated use of the H\"older inequality, Eichelsbacher and K\"onig lose track of
minimal moment assumptions: they need the finiteness of moments $\mathbb{E}[|
\xi_j^{(i)} |^r]$ with $r > ck^3$. Denisov and Wachtel \cite{DenWac} improve
their results by showing that the minimal moment requirement is actually $r =
k-1$ for $k > 3$. Since the $k$-dimensional Vandermonde determinant is a polynomial 
which has in each variable at most order $k-1$, at least moments of order $k-1$ are necessary. 
Hence Denisov and Wachtel's moment requirement is indeed minimal. For $k=3$ they need 
higher moments since their approach is based on a strong coupling which will be explained later.

A conditional version on never leaving $W^{\rm A}$ under less integrability has
not been constructed yet, and it is unclear how to do that. This is a bit
surprising since the question of leaving $W^{\rm A}$ or not has a priori nothing
to do with moments.

The case $k=2$ has already been extensively treated in the literature if one
notes that staying in order for two walkers can be translated to staying
positive for a single walker. See for example Bertoin and Doney \cite{BeDo}.

The paper is organised as follows. We present our main results in
Section~\ref{sec-CandD}, and the proofs are given in Section~\ref{sec-proofs}.
In the final Section~\ref{sec-alternate}, we discuss an alternate $h$-transform
for the restriction of the walk to the Weyl chamber of type C.

\section{Extension to Weyl chambers C and D}\label{sec-CandD}

A natural extension of the conditioned random walk setting just described is the
one to different Weyl chambers. They arise in Lie theory as the set of orbits of
the adjoint action on a Lie algebra or conjugation under the associated Lie
group and first appeared in connection with Brownian motion in a work by
Grabiner \cite{Gra}. There he considered Weyl chambers of type A, B, C, D, but
one should mention that the Weyl chambers B and C are actually identical. Those
of type C and D are defined as follows:
\begin{eqnarray*}
W^{\rm C} & = & \{x = (x_1, \ldots, x_k) \in \mathbb{R}^k \colon 0 < x_1 <
\ldots < x_k \}, \\
W^{\rm D} & = & \{x = (x_1, \ldots, x_k) \in \mathbb{R}^k \colon |x_1| < x_2 <
\ldots < x_k \}. 
\end{eqnarray*}
As already mentioned, the Weyl chamber of type A imposes a strict order. For
type C, an additional wall at zero occurs, and for D the mirror image of $x_1$
is incorporated into the order. It is important that these chambers are also
equipped with respective r\'{e}duites: 
$$h^{\rm D}(x) = \prod_{1\leq i < j \leq k}\left(x_j^2 - x_i^2\right) \qquad
\mbox{and} \qquad h^{\rm C}(x) = h^{\rm D}(x) \prod_{i=1}^k x_i.$$
As we can handle the two cases simultaneously, we will write Z for C or D. Of
course we need corresponding exit times
$$
\tau_x^{\rm Z} = \inf\{n\in\N_0\colon S(n) \notin W^{\rm Z} \}.
$$
One of the most important objects of this paper is the function
\begin{equation}\label{VZdef}
V^{\rm Z}(x) = h^{\rm Z}(x) - \mathbb{E}_x[h^{\rm Z}(S(\tau_x^{\rm Z}))], \qquad
x\in W^{\rm Z};
\end{equation}
its role will turn out to be analogous to the role of $V^{\rm A}$ for $W^{\rm
A}$. 

Let us formulate our assumptions on the random walk $(S(n))_{n\in\N_0}$, which are supposed for the
results of this section:
\newline
\newline
{\bf Moment Assumption (MA)}: $\mathbb{E}[|\xi_j^{(i)}|^{(r^Z)}] < \infty$,
where $r^{\rm C} = 2k-1$ and $ r^{\rm D} = 2k-2$ if $k\geq 3$, and $r^{\rm C} =
3$ and $r^{\rm D}> 2$ arbitrary in case $k=2$.
\newline
\newline
{\bf Symmetry Assumption (SA)}: $\mathbb{E}[(\xi_i^{(j)})^{r}] = 0$ for any odd
integer $r\leq r^{\rm Z}$.
\newline
\newline
{\bf Normalization Assumption (NA)}: $\mathbb{E}[(\xi_i^{(j)})^2] = 1$.
\newline
\newline
As one again sees from the definition of $h^{\rm Z}$, the moment requirements
are indeed minimal for the integrability of $h^{\rm Z}(S(n))$ in the cases
$k\geq 3$. In the case $k=2$ we need for ${\rm Z}={\rm D}$ some higher power for yet another
application of the strong coupling, since second moments do not suffice. The
assumption (SA) seems somewhat unnatural, but it will become clearer in the
proof of Proposition~\ref{mar} why we need it. The normalization assumption is
just for convenience. 

First we make an interesting observation about a martingale property:
\begin{proposition}\label{mar}
The function $h^{\rm Z}$ is regular for $(S(n))_{n\in\N_0}$, i.e., for
any $x \in \mathbb{R}^k$ we have $\mathbb{E}_x[h^{\rm Z}(S(1))] = h^{\rm Z}(x)$.
Thus, $(h^{\rm Z}(S(n)))_{n\in\N_0}$ is a martingale for any $x \in
\mathbb{R}^k$.
\end{proposition}

The proof uses the exchangeability of the step distribution of the random walk $(S(n))_{n\in\N_0}$ only, not the independence of the components. The case ${\rm Z}={\rm A}$ was treated in \cite{OCKRo}. Two important properties of $V^{\rm
Z}$ are that this function is well-defined and strictly positive on $W^{\rm Z}$.
We combine these properties with some results that are of interest in themselves: 
\begin{proposition}\label{posfin}
\begin{enumerate}
\item[a)] $V^{\rm Z}(x) = \lim_{n\rightarrow\infty} \mathbb{E}_x[h^{\rm
Z}(S(n)); \tau_x^{\rm Z} > n]$ for all $x\in W^{\rm Z}$;
\item[b)] $V^{\rm Z}$ ist monotone in the sense that $V^{\rm Z}(x) \leq V^{\rm
Z}(y)$ if $x_j - x_{j-1} \leq y_j - y_{j-1}$ for $2\leq j \leq k$
and additionally either $x_1 \leq y_1$ (Z=C) or $x_1 + x_2 \leq y_1 + y_2$
(Z=D);
\item[c)] $V^{\rm Z}(x) \sim h^{\rm Z}(x)$ in the limit $\inf_{2\leq j \leq
k}(x_j - x_{j-1})\rightarrow\infty$ together with
$x_1\rightarrow\infty$ ($Z=C$) or $(x_1 + x_2)\rightarrow\infty$ ($Z=D$)
respectively;
\item[d)] there is c positive such that $V^{\rm Z}(x) \leq c \cdot h_2^{\rm Z}(x)$ for
all $x\in W^{\rm Z}$, with
$h_t^{\rm D}(x) = \prod_{1\leq i < j \leq k}(t + | x_j - x_i |)(t + | x_j + x_i
|)$ and
$h_t^{\rm C}(x) = h_t^{\rm D}(x)\prod_{i=1}^k (t + | x_i |)$;
\item[e)] $V^{\rm Z}(x) > 0$ for all $x \in W^{\rm Z}$.
\end{enumerate}
\end{proposition}
With help of these insights we get a hold on the upper tails of the exit time:
\begin{theorem}\label{thm-upptails}
The asymptotic behavior for $n\rightarrow\infty$ of the exit time starting from
$x\in W^{\rm Z}$ is given by 
$$\mathbb{P}(\tau_x^{\rm Z} > n) \sim \varkappa^{\rm Z} V^{\rm
Z}(x)n^{-(\alpha^{\rm Z})/2}$$
with $\alpha^{\rm C} = k^2$ and $\alpha^{\rm D} = k^2 - k$, and $\varkappa^{\rm
C}$, $\varkappa^{\rm D}$ the following constants:
\begin{eqnarray*}
\varkappa^{\rm D} & = & \frac{2^{(3k^2 -3k + 2)/2}}{\pi^k k!} \prod_{1\leq i < j
\leq k} [(2j-1)^2 - (2i-1)^2]^{-1} \prod_{i=1}^k
\left[\Gamma\left(1+\frac{i}{2}\right)\Gamma\left(\frac{1+i}{2}\right)\right] \\
\varkappa^{\rm C} & = & \varkappa^{\rm D} 2^{(3k-2)/2} \prod_{i=1}^k
(2k+1-2i)^{-1}.
\end{eqnarray*}
\end{theorem}
The next result shows that $V^{\rm Z}$ is indeed suitable for an $h$-transform:
\begin{proposition}\label{reg}
$V^{\rm Z}$ is regular for the restriction of the transition kernel to $W^{\rm
Z}$. 
\end{proposition}
In particular, using Theorem~\ref{thm-upptails}, the two ways of conditioning
the walk to stay in $W^{\rm Z}$ that we mentioned in the introduction coincide.

Furthermore, we prove a functional limit theorem for the conditional walk in the
spirit of Donsker's theorem. Let us introduce the limit processes of the scaled
random walks and state our result. For a $k$-dimensional Brownian motion 
one can make a change of measure in the sense of
Doob's $h$-transform using the corresponding r\'eduite:
$$
\widehat{\mathbb{P}}_x^{(h^{\rm Z})}(B(t) \in \d y) = \mathbb{P}_x(B(t) \in \d
y; \tau_x^{\rm BM,Z} > t)\frac{h^{\rm Z}(y)}{h^{\rm Z}(x)},\qquad x,y\in W^{\rm Z},
$$
with $\tau_x^{\rm BM,Z} = \inf\{t \geq 0 \colon x + B(t) \notin W^{\rm Z} \}$
denoting the exit time of the Brownian motion from the type-Z Weyl chamber when
started at $x$. We will term the corresponding processes {\it
Dyson's Brownian Motion of type} Z; however note that for ${\rm Z}={\rm D}$ this expression is
used differently in \cite{Fer}(there is a little ambiguity, but it is not of any serious 
concern; see also \cite{KatoTane}). It is possible to start these processes from the origin 
(this can be seen by the same arguments as in \cite{Yor}). 

\begin{theorem}\label{thm-scalconv}
For $x\in W^{\rm Z}$, as $n\to\infty$,
$$
\mathbb{P}_x\left(\frac{1}{\sqrt{n}} S(n)\in\cdot\, \Big|\, \tau_x^{\rm Z} >
n\right) \Rightarrow \mu^{\rm Z}, 
$$
with $\mu^{\rm Z}$ the probability measure on $W^{\rm Z}$ with density
proportional to $h^{\rm Z}(y)\exp{(-|y|^2 /2)}$. Additionally the process
$(X^n(t))_{t\geq 0} = (\frac{1}{\sqrt{n}} S([nt]))_{t\geq 0}$ under the
probability measure $\widehat{\mathbb{P}}_{x\sqrt{n}}^{(V^{\rm Z})}$ weakly
converges to Dyson's Brownian motion of type Z started at $x$. Under 
$\widehat{\mathbb{P}}_{x}^{(V^{\rm Z})}$, this process converges weakly to 
Dyson's Brownian motion of type Z started at zero.
\end{theorem}

\section{Proofs}\label{sec-proofs}

First we prove the regularity of $h^{\rm Z}$ on $\R^k$, which is essential for
our purposes:

\begin{proof}[\bf Proof of Proposition \ref{mar}]
We make an induction on the number $k$ of components. For this we exploit the
Vandermonde determinant representation and write $h^{\rm Z}$ in the form
$$h^{\rm Z}(x)=\det\left[(x_j^{2i-2 + \gamma})_{i,j\in [k]}\right], \qquad [k] =
\{1, \ldots, k\},$$  
where $\gamma = 1$ for $Z=C$ and  $\gamma = 0$ for $Z=D$. We dispense with
another superscript as not to overburden the notation. For $k=1$ the assertion
trivially holds either by (SA) ($Z=C$) or a constant determinant ($Z=D$). Now
fix $k\geq 2$ and assume that our assertion is true for $k-1$. For any $x\in
\mathbb{R}^k$ and $m\in [k]$ we define
$$
h_m^{\rm Z}(x) =  \det\left[(x_j^{2i-2 + \gamma})_{i\in [k-1], j \in
[k]\backslash \{m\}}\right],
$$
which is the determinant of the matrix that we obtain by deleting the last row
and the $m$th column. In particular,  it is a $(k-1)$-dimensional analogue of
$h^{\rm Z}$. Using Laplace expansion we write
$$
h^{\rm Z}(x) = \sum_{m=1}^k (-1)^{m-1} x_m^{2k-2 + \gamma} h_m^{\rm Z}(x).
$$
We use this in the expectation and denote by $\mu$ the step distribution of the
random walk, to obtain
$$
\mathbb{E}_x[h^{\rm Z}(S(1))] = \int_{\mathbb{R}^k} \mu(\d y)\, h^{\rm Z}(x+y) =
\sum_{m=1}^k (-1)^{m-1} \int_{\mathbb{R}^k} \mu(\d y)\,  (x_m + y_m)^{2k-2 +
\gamma} h_m^{\rm Z}(x+y).
$$
We denote by $\nu$ the $m$-th marginal of $\mu$, which does not depend on $m$ by
exchangeability, and by $\mu_m(\d\tilde y|y_m)$ the conditional distribution of
$\mu$ given the coordinate $y_m$, which is exchangeable for $\tilde{y} = (y_1,
\ldots , y_{m-1}, y_{m+1}, \ldots, y_k)$. Hence, $\mu(\d y)=\nu(\d
y_m)\mu_m(\d\tilde y|y_m)$. By our induction hypothesis we have for any $y_m\in
\mathbb{R}$ and $x\in \mathbb{R}^k$ that 
$$\int_{\mathbb{R}^{k-1}} \mu_m(\d \tilde{y}|y_m)\, h_m^{\rm Z}(x+y) = h_m^{\rm
Z}(x).$$
This allows us to complete our computation:
\begin{align*}
\mathbb{E}_x[h^{\rm Z}(S(1))] & = \sum_{m=1}^k (-1)^{m-1} \int_{\mathbb{R}}
\nu(\d y_m)\, (x_m + y_m)^{2k-2 + \gamma} h_m^{\rm Z}(x)  \\ 
& = \sum_{m=1}^k (-1)^{m-1} \int_{\mathbb{R}} \nu(\d z)\, \sum_{l=0}^{2k-2 +
\gamma} \binom{ 2k-2 + \gamma}  l x_m^{2k-2 + \gamma - l} z^l h_m^{\rm Z}(x) \\
& =  \sum_{l=0, l\, \mbox{even}}^{2k-2+\gamma} \int_{\mathbb{R}} \nu(\d z)\, z^l
\binom{ 2k-2 + \gamma}  l \sum_{m=1}^k (-1)^{m-1} x_m^{2k-2 + \gamma - l}
h_m^{\rm Z}(x),
\end{align*}
where we used (SA) in the third line. Now we apply the Laplace expansion to the
$m$-sum in the last line. For $l\geq 2$ this $m$-sum vanishes since its summands
are equal to the determinants of matrices with two identical columns.
For $l=0$ it is equal to $h^{\rm Z}(x)$. This finishes the proof.
\end{proof}

Now we prove regularity of $V^{\rm Z}$ on $W^{\rm Z}$.

\begin{proof}[\bf Proof of Proposition \ref{reg}]
For any $x\in W^{\rm Z}$ we get by the strong Markov property and the martingale
property of Proposition~\ref{mar}
\begin{align*} 
& \mathbb{E}_x[V^{\rm Z}(S(1))1_{\{\tau_x^{\rm Z} >1\}}] = \\
& = \mathbb{E}_x[h^{\rm Z}(S(1))1_{\{\tau_x^{\rm Z} >1\}}] -
\mathbb{E}_x[\mathbb{E}_{S(1)}[h^{\rm Z}(S(\tau_x^{\rm Z}))]1_{\{\tau_x^{\rm
Z}>1\}}]\\
&=\mathbb{E}_x[h^{\rm Z}(S(1))1_{\{\tau_x^{\rm Z} >1\}}] - \mathbb{E}_x[h^{\rm
Z}(S(\tau_x^{\rm Z}))1_{\{\tau_x^{\rm Z}>1\}}]\\
&=\mathbb{E}_x[h^{\rm Z}(S(1))1_{\{\tau_x^{\rm Z} >1\}}] - \mathbb{E}_x[h^{\rm
Z}(S(\tau_x^{\rm Z}))] + \mathbb{E}_x[h^{\rm Z}(S(\tau_x^{\rm
Z}))1_{\{\tau_x^{\rm Z} \leq 1\}}]\\
&=\mathbb{E}_x[h^{\rm Z}(S(1))] - \mathbb{E}_x[h^{\rm Z}(S(1))1_{\{\tau_x^{\rm
Z} \leq 1\}}] - \mathbb{E}_x[h^{\rm Z}(S(\tau_x^{\rm Z}))] + \mathbb{E}_x[h^{\rm
Z}(S(\tau_x^{\rm Z}))1_{\{\tau_x^{\rm Z} \leq 1\}}]\\
&=V^{\rm Z}(x). 
\end{align*}
\end{proof}

Now we turn to the proofs of the remaining results, Proposition~\ref{posfin} and
Theorems~\ref{thm-upptails} and \ref{thm-scalconv}. We decided to give a sketch
only, since the methods employed by Denisov and Wachtel for the case $W^{\rm A}$
can be straightforwardly extended. Rather than recapping all technical details
in tedious length, we indicate the key steps of their proof and illuminate the
differences that are necessary to adapt. 

We first
explain how they obtain their analogues to Proposition~\ref{posfin} and
Theorems~\ref{thm-upptails} and \ref{thm-scalconv} for the Weyl chamber of type
A, in particular the asymptotics $\mathbb{P}(\tau_x^{\rm Z} > n) \sim \varkappa^{\rm
A} V^{\rm A}(x)n^{-k(k-1)/4}$, with $\varkappa^{\rm A}$ a constant, and the weak
convergence to Dyson's Brownian motion of type A. 
Their idea is to consider, additionally to $\tau_x^{\rm A}$, the stopping time 
$$
T_x^{\rm A} = \inf\{n\in\N_0\colon h^{\rm A}(S(n)) \leq 0\}.
$$
This has the advantage that the triggering of the defining condition can more
easily be exploited for estimates since we have control over the sign of $h^{\rm
A}(S(n))$. Furthermore, obviously, $T_x^{\rm A} \geq \tau_x^{\rm A}$ almost
surely. Therefore, certain estimates involving $T_x^{\rm A}$ can be directly
transferred to estimates involving $\tau_x^{\rm A}$. Crucial for their approach
is the fact that $(h^{\rm A}(S(n)))_{n\in\N_0}$ is a martingale. This yields
that the sequence $Y_n = h^{\rm A}(S(n))1_{\{T_x^{\rm A} > n\}}$ is a
nonnegative submartingale (this is solely based on the martingale property).
With this they next show that there 
is a universal constant $C$ such that 
\begin{equation}\label{unifbou}
\mathbb{E}_x[h^{\rm A}(S(n));T_x^{\rm A}>n] \leq C h_2^A(x),\qquad n\in\N, x\in W^{\rm A},
\end{equation}
where $h_t^{\rm A}(x) =
\prod_{1\leq i < j \leq k}(t + | x_j - x_i |)$ is defined similar to the expressions in 
Proposition~\ref{posfin}d). Proving \eqref{unifbou} is technical and lengthy and uses an auxiliary Weyl
chamber defined by
$$
W_{n,\epsilon}^{\rm A} = \{x \in \mathbb{R}^k\colon |x_j - x_i| > n^{1/2 -
\epsilon}, 1\leq i < j \leq k\},\qquad\epsilon > 0.
$$
If a point of $W^{\rm A}$ is additionally in $W_{n,\epsilon}^{\rm A}$, it is far
away from the boundary of $W^{\rm A}$. Furthermore, $W_{n,\epsilon}^{\rm A}$ has
the property that it is reached by the motion soon with high probability: 
the probability of the entrance time $\nu_n^{\rm A} =
\inf\{m\in\N_0\colon S(m) \in W_{n,\epsilon}^{\rm A} \}$ being bigger than
$n^{1-\epsilon}$ decays exponentially. Indeed, we have $\mathbb{P}_x(\nu_n^{\rm A} >
n^{1-\epsilon}) \leq \exp\{-Cn^{\epsilon}\}$. This can be shown by a 
subdivision of the trajectory into $n^{\epsilon}$ pieces and an application of
the central limit theorem to the pieces. Also expectations of $h^{\rm A}(S(n))$ on
the event $\{\nu_n^{\rm A} > n^{1-\epsilon}\}$ decay
exponentially, hence one can extend estimates for expectations which start from $x\in
W_{n,\epsilon}^{\rm A}$ to expectations with arbitrary starting points in
$W^{\rm A}$ by the strong Markov property. For the former one can elementarily
derive upper bounds with standard estimates like Doob's inequality. With this
one is able to prove the  bound in \eqref{unifbou}.

Now this in turn yields the integrability of $h^{\rm A}(S(\tau_x^{\rm A}))$ by a
direct application of martingale arguments. Furthermore, Denisov and Wachtel
obtain from this that the function  $V^{(T^{\rm A})}(x) =
\lim_{n\rightarrow\infty}\mathbb{E}_x[Y_n]$ is well defined on the set
$\{x\colon h^{\rm A}(x)>0\}$. For showing that $V^{\rm A}$ is strictly positive
on $W^{\rm A}$ they use that $(V^{(T^{\rm A})}(S(n))1_{\{\tau_x^{\rm A} >
n\}})_{n\in\N_0}$ is a supermartingale; again this is solely based on the
martingale property. 

Here we terminate our survey on the proofs of the
corresponding statements of d) and e) of Proposition \ref{posfin}. The
corresponding results to a) and b) pop out easily from the method of proof. The
proof of c) is actually part of the derivation of e). 
This finishes the sketch of their analogue of Proposition \ref{posfin}.

Now we turn to the sketch of the proofs of their analogues of
Theorems~\ref{thm-upptails} and \ref{thm-scalconv}. For this a coupling of
random walks and Brownian motion by Major \cite{Maj} is applied which has
already been used in other contexts, see \cite{BaiSui}, \cite{BodMar}. 

\begin{lemma}\label{lem-embedding}
Given that $\mathbb{E}[|\xi_j^{(i)}|^{2+\delta}] < \infty$ for some
$\delta\in(0,1)$, a Brownian motion $(B(t))_{t\geq0}$ can be defined on the same
probability space as the random walk $(S(n))_{n\in\N_0}$ such that, for $a \in(0,\frac{\delta}{2(2+\delta)})$,
$$
\mathbb{P}\left(\sup_{u\leq n}|S([u]) - B(u)| \geq n^{1/2 - a} \right) =
o\left(n^{2a + a\delta -\delta/2} \right).
$$
\end{lemma}

Other important tools for the proof of the asymptotic behavior of
$\mathbb{P}_x(\tau_x^{\rm A} > n)$ are estimates for the upper tails of the exit
time of Brownian motion from $W^{\rm A}$ due to Grabiner \cite{Gra} and
Varopoulos \cite{Varo}. Again, the auxiliary Weyl chamber $W_{n,\epsilon}^{\rm
A}$ is used. 

To attack the upper tails of $\tau_x^{\rm A}$, we know from the
above mentioned exponential decay of $\mathbb{P}_x(\nu_n^{\rm A} >
n^{1-\epsilon})$ that the random walk reaches $W_{n,\epsilon}^{\rm A}$ after a
short time, with high probability. Using the strong Markov property at time $\nu_n^{\rm A}$, we only
have to consider starting points $y \in W_{n,\epsilon}^{\rm A}$ 
instead of $x$. For those, we use Lemma~\ref{lem-embedding} with $a = 2\epsilon$
and see that the exit times from $W^{\rm A}$ for the Brownian motion and the coupled
random walk are roughly identical with high probability, since the distances between them,
$n^{1/2 - 2\epsilon}$, are negligible with respect to the typical distances $n^{1/2 - \epsilon}$
required in  $W_{n,\epsilon}^{\rm A}$. Hence, the upper tails of the random walk exit times
can directly be related to the ones of the Brownian motion, which are well-known.
After identifying the asymptotic behavior of $\mathbb{P}_x(\tau_x^{\rm A} >
n)$, one can use it to prove the functional limit theorem in a straightforward
manner.

So, unlike in the proof of Eichelsbacher and
K\"onig, there is no need to employ the discrete analogue of the Karlin-McGregor
formula, or H\"older's inequality; the results are derived using a comparison to
Brownian motion.

Now we argue that these proofs can be straightforwardly extended to cases C and
D. This is due to several factors. First, according to Proposition~\ref{mar},
$(h^{\rm Z}(S(n)))_{n\in\N_0}$ is also a martingale, and one can analogously
define the corresponding sub- and supermartingales, $(h^{\rm
Z}(S(n))1_{\{T_x^{\rm Z} > n\}})_{n\in\N_0}$ and $(V^{(T^{\rm
Z})}(S(n))1_{\{\tau_x^{\rm Z} > n\}})_{n\in\N_0}$. Second, as one easily sees, the
inequalities $T_x^{\rm Z} \geq \tau_x^{\rm Z}$ hold almost surely, too. Third,
for proving the estimate in d) of Proposition \ref{posfin}, we split the
functions $h^{\rm Z}$ into $h^{\rm D}(x) = \prod_{1\leq i < j \leq k}(x_j -
x_i)(x_j + x_i)$ (and $h^{\rm C}$ similarly). This is a more suitable
representation when used together with the corresponding auxiliary Weyl chambers
defined by 
\begin{eqnarray*}
W_{n,\epsilon}^{\rm D} & = & \{x \in \mathbb{R}^k\colon |x_j - x_i| > n^{1/2 -
\epsilon}, |x_j + x_i| > n^{1/2 - \epsilon}, 1\leq i < j \leq k\}, \\
W_{n,\epsilon}^{\rm C} & = & W_{n,\epsilon}^{\rm D} \cap \{x \in
\mathbb{R}^k\colon |x_i| > n^{1/2 - \epsilon}, 1\leq i \leq k \}.
\end{eqnarray*}
Again the probability of the entrance time $\nu_n^{\rm Z} = \inf\{m\in\N_0\colon
S(m) \in W_{n,\epsilon}^{\rm Z} \}$ being bigger than $n^{1-\epsilon}$ decays
exponentially, which can be proved by the same argument as for
$W_{n,\epsilon}^{\rm A}$.

Fourth, by the works of Grabiner and Varopoulos \cite{Gra,Varo}, we have
analogous estimates for the upper tails of the exit times from $W^{\rm Z}$ at
our disposal. Varopoulos formulated them more generally for conical regions (i.e.,
closed under scaling by a positive constant and addition of elements), and
Grabiner formulated them directly for the Weyl chambers:

\begin{lemma}
\begin{enumerate}
\item[a)] For all $y\in W^{\rm Z}$ we have with $\tau_y^{\rm BM,Z} = \inf\{t\geq
0\colon y + B(t) \notin W^{\rm Z}\}$,
$$
\mathbb{P}(\tau_y^{\rm BM,Z} > t) \leq C \frac{h^{\rm Z}(y)}{t^{(\alpha^{\rm
Z})/2}},\qquad t>0,
$$
where $\alpha^{\rm C} = k^2$ and $\alpha^{\rm D} = k^2 - k$.

\item[b)] As $t\to\infty$,
$$
\mathbb{P}(\tau_y^{\rm BM,Z} > t) \sim \varkappa^{\rm Z} \frac{h^{\rm
Z}(y)}{t^{(\alpha^{\rm Z})/2}},
$$
uniformly in $y \in W^{\rm Z}$ satisfying $|y|\leq \theta_t \sqrt{t}$ with some
$\theta_t\rightarrow 0$.

\item[c)] For $y\in W^{\rm Z}$, denote by $b_t^{\rm Z}(y,z)$ the density of
$\mathbb{P}(\tau_y^{\rm BM,Z} > t, y + B(t) \in \d z)$. Then, as $t\rightarrow
\infty$,
$$
b_t^{\rm Z}(y,z) \sim K^{\rm Z} t^{-\frac{k}{2}}e^{-|z|^2/(2t)}h^{\rm Z}(y)h^{\rm
Z}(z)t^{-\alpha^{\rm Z}},
$$
uniformly in $y,z \in W^{\rm Z}$ satisfying $|y|\leq \theta_t \sqrt{t}$ and $|z|
\leq \sqrt{t/\theta_t}$ with some $\theta_t\rightarrow 0$, and
$$
K^{\rm C}  = \frac{2^k k! \varkappa^{\rm
C}}{\int_{\mathbb{R}^k}e^{-|x|^2/2}|h^{\rm C}(x)|\, \d x} ,  \qquad K^{\rm D}  =
\frac{2^{k-1} k! \varkappa^{\rm D}}{\int_{\mathbb{R}^k}e^{-|x|^2/2}|h^{\rm
D}(x)|\, \d x}.
$$
\end{enumerate}
\end{lemma}

Of course we can use the same coupling of random walks and Brownian motion as in
Lemma~\ref{lem-embedding}. Using all these ingredients, we can easily adapt the
strategy employed by Denisov and Wachtel to prove Proposition~\ref{posfin} and
Theorems~\ref{thm-upptails} and \ref{thm-scalconv}.

\section{An alternate $\boldsymbol h$-transform for $\boldsymbol {W^{\rm
C}}$}\label{sec-alternate}

In this section we present another function, $\widetilde V^{\rm C}$, that is
positive and regular on the type-C Weyl chamber, $W^{\rm C}$. This function is
in general different from the function $V^{\rm C}$ defined in \eqref{VZdef}, but
because of its positivity and regularity on $W^{\rm C}$, it is amenable to an
$h$-transform of the random walk restricted to $W^{\rm C}$. This illustrates our
remark in the introduction: Not every $h$-transform of the random walk on a set
$W$ is equal to the conditional version of the walk given that it never leaves
$W$. The point is that $\widetilde V^{\rm C}$ does not necessarily govern the
upper tails of the exit time from $W^{\rm C}$, but $V^{\rm C}$ does; see
Theorem~\ref{thm-upptails}.

The idea of the construction of $\widetilde V^{\rm C}$ is to first condition
every component on staying positive and afterwards conditioning the resulting
walk on never violating the order of the components. In other words, we first
condition on never leaving $(0,\infty)^k$ and afterwards on never leaving
$W^{\rm A}$. Even though the intersection of these two sets is equal to $W^{\rm
C}$, there is no reason to hope that the sequentially conditioned random walk be
equal to the conditional walk constructed in Section~\ref{sec-CandD}; this is a
general fact about conditional probabilities.

Let us now describe the construction of $\widetilde V^{\rm C}$. For
$z\in(0,\infty)$, denote $V(z) = z - \mathbb{E}_z[S_1(\tau_z^{+,(1)})]$ with
$\tau_z^{+,(i)} = \inf\{n\in\N_0\colon S_i(n) \leq 0\}$ the exit time from
$(0,\infty)$. When we apply the method of proof of Denisov and Wachtel to $k=2$
under sufficient moment assumptions, we know that $V$ is positive and regular
for the restriction of a one-dimensional symmetric random walk
$(S_1(n))_{n\in\N_0}$ to $(0,\infty)$ and that it governs the upper tails of the
exit time $\tau_z^{+,(1)}$. By independence, $V^{\otimes k}$ is positive and
regular for the restriction of the walk $(S(n))_{n\in\N_0}$ to $(0,\infty)^k$
and governs the upper tails of the exit time $\tau_z^{+}= \inf\{n\in\N_0\colon
S(n) \notin (0,\infty)^k\}$. As a consequence,
\begin{equation}\label{condpos}
 \begin{aligned}
\widehat{\mathbb{P}}_x^{+}(S(n) \in \d y)
:&=\lim_{m\rightarrow\infty}\mathbb{P}_x(S(n)\in \d y\mid \tau_x^{+} > m)\\
& =  \mathbb{P}_x(S(n)\in \d y; \tau_x^{+} > n)\frac{V^{\otimes
k}(y)}{V^{\otimes k}(x)}.
\end{aligned}
\end{equation}
Under $\widehat{\mathbb{P}}_x^{+}$, the walk is equal to the conditional version
given that it does not leave $(0,\infty)^k$. Now we need the version of the
function $V^{\rm A}$ defined in \eqref{VAdef} for $\widehat{\mathbb{P}}_x^{+}$:
$$
V^{+,\rm A}(x)=h^{\rm A}(x)-\widehat{\mathbb{E}}_x^{+}[h^{\rm A}(S(\tau_x^{\rm
A}))],\qquad x\in W^{\rm A}.
$$

\begin{lemma} Assume that the step distribution of the walk is symmetric and
possesses finite moments of order $k-1$ for $k\geq 4$ or of some order $r > 2$
in cases $k=3$ and $k=2$. Then the function 
 $$
\widetilde V^{\rm C}=V^{+,\rm A} V^{\otimes k}
$$
is positive in  $W^{\rm C}$ and regular for the restriction of the transition
kernel to $W^{\rm C}$.
\end{lemma}

\begin{proof}
The $h$-transform of $\widehat{\mathbb{P}}_x^{+}$ with $V^{+,\rm A}$ on $W^{\rm
A}$ is equal to the conditional version given that the walk does not leave
$W^{\rm A}$, i.e.,
\begin{equation}\label{condA}
\lim_{m\rightarrow\infty}\widehat{\mathbb{P}}_x^{+}(S(n)\in \d y\,|\,
\tau_x^{\rm A} > m)  =  \widehat{\mathbb{P}}_x^{+}(S(n)\in \d y; \tau_x^{\rm A}
> n)\frac{V^{+,\rm A}(y)}{V^{+,\rm A}(x)}.
\end{equation}
Using \eqref{condpos}, we see that
$$
\widehat{\mathbb{P}}_x^{+}(S(n)\in \d y; \tau_x^{\rm A} >
n)=\mathbb{P}_x(S(n)\in \d y; \tau_x^{+} > n,\tau_x^{\rm A} > n)\frac{V^{\otimes
k}(y)}{V^{\otimes k}(x)}.
$$
Using this in \eqref{condA} and noting that $\{\tau_x^{+} > n,\tau_x^{\rm A} >
n\}=\{\tau_x^{\rm C} > n\}$, we arrive at
$$
\begin{aligned}
\lim_{m\rightarrow\infty}\widehat{\mathbb{P}}_x^{+}(S(n)\in \d y\,|\,
\tau_x^{\rm A} > m)
&=  \mathbb{P}_x(S(n)\in \d y;  \tau_x^{\rm C} > n)\frac{V^{+,\rm
A}(y)}{V^{+,\rm A}(x)}\frac{V^{\otimes k}(y)}{V^{\otimes k}(x)}\\
&=\mathbb{P}_x(S(n)\in \d y;  \tau_x^{\rm C} > n)\frac{\widetilde V^{\rm
C}(y)}{\widetilde V^{\rm C}(x)}.
\end{aligned}
$$
Since the left hand side is a probability measure in $y\in W^{\rm C}$, the right
hand side is as well. This shows the regularity of $\widetilde V^{\rm C}$ in 
$W^{\rm C}$. The positivity is obvious.
\end{proof}

For the case D this approach does not work since it is not clear how to 
divide the condition in $W^{\rm D}$ into two conditions that can separately
be handled with the methods presented in this paper.
\newline
\newline
{\bf Acknowledgements}

P. Schmid would like to thank P. Ferrari and T. Sasamoto for helpful discussions.

\bigskip

{\sc Patrick Schmid}, Universit\"at Leipzig, Mathematisches Institut, Postfach
100920, D-04009 Leipzig, Germany,
\newline
{\tt Patrick.Schmid@math.uni-leipzig.de  }

\medskip

{\sc Wolfgang K\"onig}, Technical University Berlin, Str. des 17. Juni 136,
10623 Berlin, and Weierstra\ss\ Institute for Applied Analysis and Stochastics,
Mohrenstr. 39, 10117 Berlin, Germany
\newline
{\tt koenig@math.tu-berlin.de, koenig@wias-berlin.de}

\end{document}